\documentclass[twoside,leqno,10pt]{amsart}
\usepackage{amsfonts}
\usepackage{amsmath}
\usepackage{amscd}
\usepackage{amssymb}
\usepackage{amsthm}
\usepackage{amsrefs}
\usepackage{latexsym}
\usepackage{bbm}
\begin{document}

\newtheorem{theorem}[subsection]{Theorem}
\newtheorem{proposition}[subsection]{Proposition}
\newtheorem{lemma}[subsection]{Lemma}
\newtheorem{corollary}[subsection]{Corollary}
\newtheorem{conjecture}[subsection]{Conjecture}
\newtheorem{prop}[subsection]{Proposition}
\numberwithin{equation}{section}
\newcommand{\dif}{\mathrm{d}}
\newcommand{\intz}{\mathbb{Z}}
\newcommand{\ratq}{\mathbb{Q}}
\newcommand{\natn}{\mathbb{N}}
\newcommand{\comc}{\mathbb{C}}
\newcommand{\rear}{\mathbb{R}}
\newcommand{\prip}{\mathbb{P}}
\newcommand{\uph}{\mathbb{H}}
\newcommand{\fief}{\mathbb{F}}
\newcommand{\majorarc}{\mathfrak{M}}
\newcommand{\minorarc}{\mathfrak{m}}
\newcommand{\sings}{\mathfrak{S}}
\newcommand{\fA}{\ensuremath{\mathfrak A}}
\newcommand{\mn}{\ensuremath{\mathbb N}}
\newcommand{\half}{\tfrac{1}{2}}
\newcommand{\f}{f\times \chi}
\newcommand{\summ}{\mathop{{\sum}^{\star}}}
\newcommand{\chiq}{\chi \bmod q}
\newcommand{\chidb}{\chi \bmod db}
\newcommand{\chid}{\chi \bmod d}
\newcommand{\sym}{\text{sym}^2}
\newcommand{\hhalf}{\tfrac{1}{2}}
\newcommand{\sumstar}{\sideset{}{^*}\sum}
\newcommand{\sumprime}{\sideset{}{'}\sum}
\newcommand{\sumprimeprime}{\sideset{}{''}\sum}
\newcommand{\shortmod}{\ensuremath{\negthickspace \negthickspace \negthickspace \pmod}}
\newcommand{\V}{V\left(\frac{nm}{q^2}\right)}
\newcommand{\sumi}{\mathop{{\sum}^{\dagger}}}
\newcommand{\mz}{\ensuremath{\mathbb Z}}
\newcommand{\leg}[2]{\left(\frac{#1}{#2}\right)}

\title{One level density of low-lying zeros of families of $L$-functions}
\date{\today}
\author{Peng Gao and Liangyi Zhao}
\maketitle

\begin{abstract}
In this paper, we prove some one level density results for the
low-lying zeros of families of $L$-functions.  More specifically,
the families under consideration are that of $L$-functions of
holomorphic Hecke eigenforms of level 1 and weight $k$ twisted
with quadratic Dirichlet characters and that of cubic and quartic
Dirichlet $L$-functions.
\end{abstract}



\tableofcontents

\section{Introduction}

The density conjecture of N. Katz and P. Sarnak \cites{KS1, K&S} suggests that the
distribution of zeros near the central point of a family of
$L$-functions is the same as that of eigenvalues near $1$ of a
corresponding classical compact group.  This has been confirmed
for many families of $L$-functions, such as different types of
Dirichlet $L$-functions \cites{Gao, O&S, Ru, Miller1, HuRu}, $L$-functions with
characters of the ideal class group of the imaginary quadratic
field $\ratq(\sqrt{-D})$ \cite{FI}, automorphic $L$-functions \cites{ILS, DuMi2, HuMi, RiRo, Royer}, elliptic curves $L$-functions \cites{SBLZ1, HB1, Brumer, SJM, Young}, symmetric powers of $GL(2)$ $L$-functions \cite{Gu2, DM} and a family of $GL(4)$ and $GL(6)$
$L$-functions \cite{DM}.  Literatures in this direction are too
numerous to be completely listed here.  In this paper, we prove
some one level density results for the low-lying zeros of families
of $L$-functions of holomorphic Hecke eigenforms of level 1 and
weight $k$ twisted with quadratic Dirichlet characters and those
of families of cubic and quartic Dirichlet $L$-functions.
\newline

Let $\chi$ be a primitive Dirichlet character and we denote
the non-trivial zeroes of the Dirichlet $L$-function
   $L(s, \chi)$ by $\half+i \gamma_{\chi, j}$.  Without assuming the generalized Riemann hypothesis (GRH), we order them as
\begin{equation}
\label{zeroorder}
    \ldots \leq
   \Re \gamma_{\chi, -2} \leq
   \Re \gamma_{\chi, -1} < 0 \leq \Re \gamma_{\chi, 1} \leq \Re \gamma_{\chi, 2} \leq
   \ldots.
\end{equation}
    For any primitive Dirichlet character $\chi$
    of conductor $q$ of size $X$, we set
    \[ \tilde{\gamma}_{\chi, j}= \frac{\gamma_{\chi, j}}{2 \pi} \log X \]
and define for an even Schwartz class function $\phi$,
\begin{equation} \label{Sdef}
S(\chi, \phi)=\sum_{j} \phi(\tilde{\gamma}_{\chi, j}).
\end{equation}

In \cite{O&S},  A. E. \"{O}zluk and C. Snyder studied the family
of quadratic Dirichlet $L$-functions.  Consider the family of
Dirichlet $L$-functions of the form $L(\chi_{8d},s)$ for $d$ odd
and square-free with $X\le d \le 2X$, where
$\chi_{8d}=(\frac{8d}{\cdot})$ is the Kronecker symbol.  Let
$D(X)$ denote the set of such $d$'s. It's easy to see, as noted in \cite{Gao}, that
 \[ \# D(X) \sim \frac{4X}{\pi^2}. \]
Assuming GRH for this family, it follows from the result of
\"{O}zluk and Snyder \cite{O&S} that for
$W_{USp}(x)=1-\frac{\sin(2\pi x)}{2\pi x}$, we have
\begin{align}
\label{1} \frac{1}{ \# D(X)}\sum_{d \in D(X)} S(\chi_{8d}, \phi)
\sim \int_{\rear} \phi(x) W_{USp}(x) \dif x,
\end{align}
as $X \rightarrow \infty$ provided that the support of $\hat{\phi}$, the Fourier
transform of $\phi$, is contained in the interval $(-2, 2)$.  The expression on the left-hand side of \eqref{1} is known as the one-level density of the low-lying zeros for this family of $L$-functions under consideration.\newline

The kernel of the integral $W_{USp}$ in \eqref{1} is the same function which occurs
on the random matrix theory side, when studying the eigenvalues of
unitary symplectic matrices. This shows that the family of
quadratic Dirichlet $L$-functions is a symplectic family.
  In \cite{Ru}, M. O. Rubinstein studied all the $n$-level
  densities of the low-lying zeros of the families of
quadratic twists of $L$-functions attached to a
self-contragredient automorphic cuspidal representation, as well as
the family of quadratic Dirichlet $L$-functions.  He showed that
they converge to the symplectic densities for test functions
  $\phi(x_1,\ldots, x_n)$ whose Fourier transforms $\hat{\phi}(u_1,\ldots, u_n)$ have their supports contained in the set
\begin{equation*}
 \left\{ (u_1, \cdots, u_n) \in \rear^n :   \sum^n_{i=1}|u_i|<1 \right\}.
 \end{equation*}
The result of Rubinstein does not assume
  GRH.  In \cite{Gao}, assuming the truth of GRH, the first-named author extended Rubinstein's result \cite{Ru} and showed that it holds for $\hat{\phi}(u_1,\ldots, u_n)$ with support in the set
\[ \left\{ (u_1, \cdots, u_n) \in \rear^n :   \sum^n_{i=1}|u_i|<2 \right\}. \]

In this paper, we consider a few other families of $L$-functions. First, let $f$ be a fixed
holomorphic Hecke eigenform of level $1$ and weight $k$. For
$\Im(z)>0$ we have a Fourier expansion of $f$:
\begin{align*}
f(z)=\sum_{n=1}^{\infty} a_f(n)n^{\frac{k-1}{2}}e (n z),
\end{align*}
where the coefficients $a_f(n)$ are real and normalized with $a_f(1)=1$ and satisfy the
Ramanujan-Petersson bound
\[ |a(n)|\le d(n), \]
with $d(n)$ being the divisor function. Let $\chi$ be a primitive Dirichlet character
of conductor $q$.  The $L$-function of the twist of $f$ by
$\chi$ is given by
\begin{align}
\label{euler} L(\f,s)= \sum_{n=1}^{\infty}
\frac{a_f(n)\chi(n)}{n^s}=\prod_{p}
\left(1-\frac{a_f(p)\chi(p)}{p^{s}}+\frac{\chi(p^2)}{p^{2s}}\right)^{-1}
\end{align}
for $\Re(s)>1$. This $L$-function continues to an entire function
and satisfies the functional equation
\begin{equation*}
\left( \frac{q}{2\pi} \right)^{s}\Gamma \left(\frac{k}{2}+s \right)L \left( f,\frac{1}{2} +s \right)=\imath_{\chi} \left( \frac{q}{2\pi} \right)^{-s}\Gamma \left( \frac{k}{2}-s \right )L \left( f\times
\overline{\chi},\frac{1}{2}-s \right).
\end{equation*}
where $\imath_{\chi}=i^k \tau(\chi)^2/q$ and $\tau(\chi)$
is the Gauss sum associated to $\chi$ (thus $|\imath_{\chi}|=1$).
cf. \cite[Prop. 14.20]{iwakow}.  See \cites{iwakow, HI3} for detailed discussions of Hecke eigenforms. \newline

  For a fixed $f$, we consider the family of quadratic twists of
  $L$-functions $L(f \times \chi_{8d},s)$ for $d$
odd and square-free with $X\le d \le 2X$. For this family,
Rubinstein \cite{Ru} has shown that the $n$-level densities of the
low-lying zeros converge to the orthogonal densities for test
functions $\phi(x_1,\ldots, x_n)$ whose Fourier transforms
$\hat{\phi}(u_1,\ldots, u_n)$ are supported in the set
\begin{equation*}
 \left\{ (u_1, \cdots, u_n) \in \rear^n :   \sum^n_{i=1}|u_i|< \frac{1}{2} \right\}.
 \end{equation*}
  More
precisely, we denote the zeroes of $L(f \times \chi_{8d},s)$ by
$\half+i \gamma_{f, 8d, j}$, and order them in a manner similar to \eqref{zeroorder}.  Let $\tilde{\gamma}_{f, 8d,j}= \gamma_{f,
8d,j} 2\log X/(2\pi) $ and define for an even Schwartz function
$\phi$,
\begin{align*}
D(d, f, \phi)=\sum_{j} \phi(\tilde{\gamma}_{f, 8d,j}).
\end{align*}
  We set
  \[ W_{SO+}(x)=1+\frac{\sin(2\pi x)}{2\pi x}, \; W_{SO-}(x)=\delta(x)+1-\frac{\sin(2\pi x)}{2\pi x} \; \mbox{and} \;
  W_{O}(x)=\frac{1}{2} \left( W_{SO+}(x)+W_{SO-}(x) \right) , \]
  where $\delta_0(x)$
  is the Dirac distribution at $x=0$.   Note that the three orthogonal densities
are indistinguishable for test functions whose Fourier transforms
are supported in $(-1, 1)$. The result of Rubinstein asserts that
\begin{align*}
 \frac{1}{\# D(X)}\sum_{d \in D(X)} D(d, f, \phi) \sim
\int_{\mathbb{R}} \phi(x) \left( 1+\frac{\sin(2\pi x)}{2\pi x} \right) \dif x,
\end{align*}
as $X \rightarrow \infty$ provided that $\hat{\phi}$ is supported
on the interval $(-1/2, 1/2)$. This shows that the family $L(f
\times \chi_{8d},s)$ has orthogonal symmetry.
\newline

  In this paper, we improve the above-mentioned result of Rubinstein by
  doubling the size of the allowable support of the Fourier transform of the test
  function in the case of the one level density. For technical reasons, we consider the average over the family by a smooth
  weight. Let $\Phi_X(t)$ be a non-negative smooth function supported on $(1,2)$,
    satisfying $\Phi_X(t)=1$ for $t \in (1+1/U, 2-1/U)$ with $U=\log X$ and such that
    $\Phi^{(j)}_X(t) \ll_j U^j$ for all integers $j \geq 0$. Our result is
\begin{theorem} \label{quadraticmainthm}
Suppose that GRH is true. Let $f$ be a fixed holomorphic Hecke
eigenform of level $1$ and weight $k$. Let $\phi(x)$ be an even
Schwartz function whose Fourier transform $\hat{\phi}(u)$ has
compact support in $(-1, 1)$.  Then we have
\begin{align} \label{theo1eq}
 \lim_{X \rightarrow +\infty}\frac{1}{\# D(X)}\sum_{d \in D(X)} \Phi_X \left( \frac {d}{X} \right) D(d, f, \phi)
 = \int_{\mathbb{R}} \phi(x) \left(1+\frac{\sin(2\pi x)}{2\pi x} \right) \dif x.
\end{align}
\end{theorem}

We shall describe briefly the proof of
Theorem~\ref{quadraticmainthm} in this paragraph. Using a modified
version of the explicit formula, Lemma~\ref{lem2.2.1},
$D(d,f,\phi)$ in \eqref{theo1eq}, a sum over the zeros of the
$L$-function under consideration, is converted into a sum over
primes and prime powers.  The most important among these will be
the sums over primes (the one appearing in \eqref{realprimeSdef})
and prime squares, after showing that the contributions of the
higher prime powers are negligible.   The sum involving prime
squares can be easily handled.  Thus far, each of the sums can be
disposed of for each individual $L$-functions in question, without
appealing to the additional averaging over $d$ in \eqref{theo1eq}.
But this additional sum over the family is needed in the treatment
of the more difficult sum involving primes, \eqref{realprimeSdef}.
Following a method of K. Soundararajan \cite{sound1}, after
detecting the square-free condition of $d$ using the M\"obius
function and applying Poisson summation formula to the sum over
$d$, we are led to consider sums involving quadratic Gauss sums,
which is further split into pieces according to the size of the
relevant parameters.  GRH is needed to estimate a component of
these sums. Other estimates are quoted from \cites{Miller1} to
give the final estimate for the sum in \eqref{realprimeSdef}.
Theorem~\ref{quadraticmainthm} follows after combining all
estimates.\newline

We note that the one level density of the low-lying zeroes of the
family of quadratic twists of Hecke $L$-functions evaluated in
Theorem \ref{quadraticmainthm} is insufficient to determine which
of the orthogonal symmetry types, $SO(\text{even})$, $O$ or
$SO(\text{odd})$, is attached of each family,
due to the small support-restriction on the Fourier transforms of
test functions.  However, as shown in \cites{DM, SJM}, the two level
density allows one to distinguish between the three orthogonal
symmetries for test functions whose Fourier transforms have
arbitrarily small support.  In fact, it follows from Rubinstein's result on two-level density of this family (Lemma 7 in \cite{Ru} and also the formula on the top of page 179 in \cite{Ru}) that the symmetry type attached to the family of quadratic twists of $L$-functions $L(f \times \chi_{8d},s)$ for $d$ odd and square-free with $X\le d \le 2X$ is $SO(\text{even})$.  \newline

In addition, we consider the one level density of the low-lying
zeros of the families of cubic and quartic Dirichlet $L$-functions
in this paper. \newline

For a primitive cubic Dirichlet character $\chi$ of conductor $q$ coprime to $3$, it is shown in \cite{B&Y} that $q$ must be square-free and a product of primes congruent to $1$ modulo $3$. It follows that $\chi$ is a product of primitive cubic characters modulo the prime divisors of $q$ and for each prime divisor $p$ of $q$, there are exactly two primitive characters with conductor $p$, each being the square (also the complex conjugate) of the other. \newline

We shall prove the following:
\begin{theorem}
\label{mainthm}
   Let $f(x)$ be an even Schwartz function whose Fourier transform $\hat{f}(u)$ has compact
support in $(-3/7, 3/7)$, then
\begin{align*}
 \lim_{X \rightarrow +\infty}\frac{1}{ \# C(X) }\sum_{X \leq q \leq 2X}\ \sumstar_{\substack{\chi \shortmod{q} \\ \chi^3 = \chi_0}} S(\chi, f)
 = \int_{\mathbb{R}} f(x) \dif x.
\end{align*}
Here the ``$*$'' on the sum over $\chi$ means that the sum is
restricted to primitive characters and $C(X)$ denotes the set of
primitive cubic characters of conductor $q$ not divisible by 3 and
$X \leq q \leq 2X$.
\end{theorem}

We point out here that Theorem \ref{mainthm} is analogous to a
result of A. M. G\"ulo\u{g}lu \cite{Gu} and that his result depends
on GRH while Theorem~\ref{mainthm} does not.  G\"ulo\u{g}lu
\cite{Gu} studied the one level density of the low-lying zeros of
a family of Hecke $L$-functions associated with cubic symbols
$\chi_c=(\frac {\cdot}{c})_3$ with $c$ square-free and congruent
to 1 modulo 9, regarded as primitive ray class characters of the
ray class group $h_{(c)}$.  We recall here that for any $c$, the
ray class group $h_{(c)}$ is defined to be $I_{(c)}/P_{(c)}$,
where $I_{(c)} = \{ \mathcal{A} \in I, (\mathcal{A}, (c)) = 1 \}$
and $P_{(c)} = \{(a) \in P, a \equiv 1 \pmod{c} \}$ with $I$ and
$P$ denoting the group of fractional ideals in $K=\ratq (\omega)$
and the subgroup of principal ideals, respectively.  Here $\omega
= \exp (2 \pi i/3)$.  The Hecke $L$-function associated with
$\chi_c$ is defined for $\Re(s) > 1$ by
\begin{equation*}
  L(s, \chi_c) = \sum_{0 \neq \mathcal{A} \subset
  O_K}\chi_c(\mathcal{A})(N(\mathcal{A}))^{-s},
\end{equation*}
  where $\mathcal{A}$ runs over all non-zero integral ideals in $K$ and $N(\mathcal{A})$ is the
norm of $\mathcal{A}$. As shown by E. Hecke, $L(s, \chi_c)$ admits
analytic continuation to an entire function and satisfies a
functional equation.  We refer the reader to \cites{Gu, Luo} for a more detailed discussion of these Hecke characters and $L$-functions.  We denote non-trivial zeroes of $L(s,
\chi_c)$ by $\half+i \gamma_{\chi_c, j}$ and order them in a
fashion similar to \eqref{zeroorder}.  Let $C_{(9)}(X)$ stand for
the set of $\chi_c$'s with $c$ square-free, congruent to 1 modulo
9 and $X \leq N(c) \leq 2X$. We define $S(\chi_c, f)$ similarly as
$S(\chi, \phi)$ in \eqref{Sdef}. \newline

For the family of Hecke $L$-functions considered by G\"ulo\u{g}lu
\cite{Gu}, we can apply our approach to Theorem \ref{mainthm} to
obtain the following result without assuming GRH:
\begin{theorem}
\label{mainthmonK} Let $f(x)$ be an even Schwartz function whose
Fourier transform $\hat{f}(u)$ has compact support in $(-3/5,
3/5)$, then
\begin{align*}
 \lim_{X \rightarrow +\infty}\frac{1}{\# C_{(9)}(X)}\sumstar_{\substack{c \equiv 1 \shortmod{9} \\ X \leq N(c) \leq 2X}} S(\chi_c, f)
 = \int_{\mathbb{R}} f(x) \dif x.
\end{align*}
   Here the ``$*$'' on the sum over $c$ means that the sum is restricted to square-free elements $c$ of $\mathbb{Z}[\omega]$ .
\end{theorem}

The support obtained under GRH by G\"ulo\u{g}lu \cite{Gu} is $(-31/30, 31/30)$ in place of $(-3/5,3/5)$ in our unconditional Theorem ~\ref{mainthmonK}.   In addition, the result G\"ulo\u{g}lu's result is a smoothed version of one-level density and depends on other results that require this smoothness.  We also note that, as
\[ \int_{\mathbb{R}} f(x) \dif x=\int_{\mathbb{R}} f(x)W_U(x) \dif x, \; \mbox{with} \; W_U(x)=1, \] Theorems \ref{mainthm} and \ref{mainthmonK} show that the family of cubic Dirichlet $L$-functions as well as the family of Hecke $L$-functions  associated with cubic symbols are unitary families, an observation made in \cite{Gu}.  \newline

Analogous to Theorem \ref{mainthm}, we also study the one level
density of the low-lying zeros of the family of quartic Dirichlet
$L$-functions. We have the following
\begin{theorem}
\label{quarticthm} Let $f(x)$ be an even Schwartz function whose
Fourier transform $\hat{f}(u)$ has compact support in $(-3/7,
3/7)$, then
\begin{align*}
 \lim_{X \rightarrow +\infty}\frac{1}{\# Q(X)}\sum_{X \leq q \leq 2X}\sumstar_{\substack{\chi \shortmod{q} \\ \chi^4 = \chi_0, \chi^2 \neq \chi_0}} S(\chi, f)
 = \int_{\mathbb{R}} f(x) \dif x.
\end{align*}
   Here the ``$*$'' on the sum over $\chi$ means that the sum is restricted to primitive
   characters and $Q(X)$ denotes the set of primitive complex quartic characters with odd conductor q and $X \leq q \leq 2X$.
\end{theorem}

The proofs of Theorems \ref{mainthm}, \ref{mainthmonK} and
\ref{quarticthm} are similar.  Therefore, we describe briefly here
only the proof of Theorem~\ref{mainthm}.  Similar to the proof of
Theorem~\ref{quadraticmainthm}, we first convert the sum over
zeros of the $L$-functions under consideration into sums over
primes using the relevant versions of the explicit formula.  This
leads us to consider sums involving the cubic or quartic symbols
over primes and prime squares.  These are then rewritten as sums
involving Hecke characters with new summation conditions.  The
summation conditions are handled using M\"obius inversion and the
final estimates are obtained using a P\'olya-Vinogradov type bound
for these Hecke characters, Lemma~\ref{PVbounds}.\newline

\subsection{Notations} The following notations and conventions are used throughout the paper.\\
\noindent $e(z) = \exp (2 \pi i z) = e^{2 \pi i z}$. \newline $f =
O(g)$ or $f \ll g$ means $|f| \leq cg$ for some unspecified
positive constant $c$. \newline
For $x \in \rear$, $\| x \| =
\min_{n \in \intz} |x-n|$ denotes the distance between $x$ and the
closest integer.

\section{Preliminaries}
\label{sec 2}
\subsection{The Explicit Formula}

Our approach in this paper relies on the following explicit
formula, which essentially converts the sum over zeros of an
$L$-function to the sum over primes.
\begin{lemma}
\label{lem2.2.1}
   Let $f(x)$ be an even Schwartz function whose Fourier transform
   $\hat{f}(u)$ is compactly supported.  Then for any primitive Dirichlet character
   $\chi$, we have
\begin{equation} \label{01.3}
\begin{split}
S(\chi, f) =\int\limits^{\infty}_{-\infty}& f(t) \dif t-\frac
1{\log X}\sum_{p}\frac {\log p}{\sqrt{p}}\hat{f}\left( \frac {\log
p}{\log X}\right) \left(\chi(p) + \overline{\chi}(p) \right ) \\
&-\frac 1{\log X}\sum_{p}\frac {\log p}{p}\hat{f} \left( \frac {2\log
p}{\log X}\right) \left(\chi(p^2) + \overline{\chi}(p^2) \right )
+O\left(\frac{1}{\log X}\right).
\end{split}
\end{equation}
\end{lemma}

\begin{proof}
We combine \cite[(2.16)]{R&S}, the fact that
     $f(x)$ is rapidly decreasing and the Stirling formula which gives that $\Gamma'/\Gamma(s)=\log s + O
    (1/|s|)$, uniformly for $|\arg s| \leq \pi - \delta, |s| \geq
    1$ to replace the $\Gamma'/\Gamma$-terms in  \cite[(2.16)]{R&S} by $O(1/\log
    X)$.  Moreover, the terms $n=p^k$ for $k \geq 3$, $p$ prime in the sum on the right-hand side of
    \cite[(2.16)]{R&S} contribute
\begin{equation*}
   \ll \sum_{p^k, k \geq 3}\frac
{\log p }{p^{k/2}} \ll 1.
\end{equation*}
The lemma follows from these observations.
\end{proof}

    Similarly, we have \cite[Lemma 4.1]{Gu}:
\begin{lemma}
\label{lem2.4}
   Let $f(x)$ be an even Schwartz function whose Fourier transform
   $\hat{f}(u)$ has compact support.  Then for any square-free $c \equiv 1 \hspace{0.1in} \shortmod{9}$ of
   $\mathbb{Z}[\omega]$, we have
\begin{equation*}
\begin{split}
S(\chi_c, f) =\int\limits^{\infty}_{-\infty} & f(t) \dif t-\frac
1{\log X}\sum_{\mathfrak{p}}\frac {\log
N(\mathfrak{p})}{\sqrt{N(\mathfrak{p})}}\hat{f}\left( \frac {\log
N(\mathfrak{p})}{\log X} \right) \left(\chi_c(\mathfrak{p}) + \overline{\chi}_c(\mathfrak{p}) \right ) \\
&-\frac 1{\log X}\sum_{\mathfrak{p}}\frac {\log
N(\mathfrak{p})}{N(\mathfrak{p})}\hat{f} \left( \frac {2\log
N(\mathfrak{p})}{\log X}\right) \left(\chi_c(\mathfrak{p}^2) +
\overline{\chi}_c(\mathfrak{p}^2) \right ) +O\left(\frac{1}{\log
X}\right),
\end{split}
\end{equation*}
   where $\mathfrak{p}$ runs over all non-zero integral ideals in
   $\ratq (\omega)$.
\end{lemma}

   For $\Re(s)>1$, we can rewrite the Euler product \eqref{euler} of
   $L(f \times \chi_{8d}, s)$ as
\begin{align*}
L(f \times \chi_{8d},s) = \prod_{p}
\left(1-\frac{\alpha_{f}(p)\chi_{8d}(p)}{p^{s}}\right)^{-1}\left(1-\frac{\alpha^{-1}_{f}(p)\chi_{8d}(p)}{p^{s}}\right)^{-1},
\end{align*}
    with $\alpha_{f}(p)+\alpha^{-1}_{f}(p)=a_f(p)$. Similar to
    Lemma \ref{lem2.2.1}, we have the following
\begin{lemma}
\label{lem2.1}
   Let $\phi(x)$ be an even Schwartz function whose Fourier transform
   $\hat{\phi}(u)$ is compactly supported.  Then for $d \in D(X)$,
   we have
\begin{align}
\label{2.5}
   D(d, f, \phi)   =\int\limits^{\infty}_{-\infty}\phi(t) \dif t+\frac 1{2}
   \int\limits^{\infty}_{-\infty}\hat{\phi}(u) \dif u-S(f, d,X; \hat{\phi})+O \left( \frac {\log \log 3X}{\log
   X} \right),
\end{align}
    with the implicit constant depending on $\phi$ and
\begin{equation} \label{realprimeSdef}
   S(f, d,X;\hat{\phi})=\frac 1{\log X}\sum_p\frac {a_f(p)\log
   p}{\sqrt{p}}\left( \frac{8d}{p} \right) \hat{\phi}\left( \frac {\log p}{2\log X} \right).
\end{equation}
\end{lemma}
\begin{proof}
    We use \cite[(2.16)]{R&S} again here and identify $c_{\pi}(n)$ in \cite[(2.16)]{R&S} as $\Lambda(n)\chi_{8d}(n)a_f(n)$
    with $a_f(p^k)=\alpha^k_f(p)+\alpha^{-k}_f(p)$ in our case.
    By Deligne's proof \cite{D} of the Weil conjecture, we know that
   $|\alpha_f(p)|=1$ so that the terms $n=p^k$ for $k \geq 3$, $p$ prime in the sum on the right-hand side of
    \cite[(2.16)]{R&S} contribute $O(1)$. Moreover, note that
\begin{equation*}
  \sum_{p | 8d}\frac {\log p}{p} \ll \log \log 3X.
\end{equation*}

    Thus the terms $n=p^2$ in the sum on the right-hand side of
    \cite[(2.16)]{R&S} contribute
\begin{align*}
 \frac 1{\log X}\sum_{p}\frac {a_f(p^2)\log p}{p}
   \hat{\phi}\left( \frac {\log p}{\log X} \right)+ O \left( \frac {\log \log 3X}{\log
   X} \right).
\end{align*}
Recall that
\begin{equation} \label{apsquare}
a_f(p^2)=a^2_f(p)-2
\end{equation}
and it follows from Proposition 2.3 in \cite{R&S} that
\begin{equation}
\label{2.90}
    \sum_{p \leq x}\frac {a^2_f(p) \log^2 p }{p}= \frac {\log^2 x}{2}+O(\log
    x).
\end{equation}
Note also that we have Mertens' formula \cite[p. 57]{Da},
\begin{equation}
\label{mer}
    \sum_{p \leq x}\frac {\log p}{p}= \log x+O(1).
\end{equation}
Combining \eqref{apsquare}, \eqref{2.90} and \eqref{mer}, we deduce that
\begin{equation*}
    \sum_{p \leq x}\frac {a_f(p^2)\log^2 p}{p}= -\frac {\log^2 x}{2}+O(\log x).
\end{equation*}
From the above and partial summation, we get that
\begin{eqnarray*}
\frac 1{\log X} \sum_{p}\frac {a_f(p^2) \log p}{p}
   \hat{\phi}\left( \frac {\log p}{\log X} \right) & = & -\frac 1{\log X}\int\limits^{\infty}_{1}\hat{\phi}\left( \frac {\log t}{\log X}\right) \frac {\dif t}{t}+O\left( \frac {\log \log X}{\log X} \right) \\
  &=& -\frac 1{2}
\int\limits^{\infty}_{-\infty}\hat{\phi}(t)\dif t+O \left( \frac {\log \log X}{\log
    X} \right).
\end{eqnarray*}
  The assertion of the lemma follows from this easily.
\end{proof}
\subsection{Poisson Summation}
   We now fix $f$ and $X$ and henceforth write $\Phi(t)$ in place
   of $\Phi_X(t)$.  We shall focus on finding the asymptotic expression of (with $n \geq 1$)
\begin{equation*}
    S(X,Y; \hat{\phi}, f, \Phi) :=
    \sum_{\gcd(d,2)=1}\mu^2(d)\sum_{p \leq Y}\frac {a_f(p)\log p}{\sqrt {p}} \left( \frac{8d}{p} \right) \hat{\phi} \left( \frac {\log p}{2\log X} \right) \Phi \left( \frac {d}{X} \right),
\end{equation*}
    where $\hat{\phi}(u)$ is smooth and has its support contained in the interval $(-1+\epsilon, 1-\epsilon)$ for any $\epsilon>0$. To emphasize this condition, here
and throughout, we shall set $Y=X^{2-2\epsilon}$ and write the
condition $p \leq Y$ explicitly. \newline

Let $Z >1$ be a real parameter to be chosen later and write
     $\mu^2(d)=M_Z(d)+R_Z(d)$ where
\begin{equation*}
    M_Z(d)=\sum_{\substack {l^2|d \\ l \leq Z}}\mu(l) \; \mbox{and} \;  R_Z(d)=\sum_{\substack {l^2|d \\ l >
    Z}}\mu(l).
\end{equation*}

    Define
\[ S_M(X,Y; \hat{\phi},f, \Phi) = \sum_{\gcd(d,2)=1}M_Z(d)\sum_{p \leq Y}
    \frac {a_f(p)\log p}{\sqrt {p}} \left( \frac{8d}{p} \right) \hat{\phi} \left( \frac {\log p}{2\log X} \right) \Phi\left( \frac{d}{X} \right) \]
    and
\[ S_R(X,Y; \hat{\phi}, f, \Phi)
=\sum_{\gcd(d,2)=1}R_Z(d)\sum_{p \leq Y}
    \frac {a_f(p)\log p}{\sqrt {p}}\left( \frac{8d}{p} \right) \hat{\phi} \left( \frac {\log p}{2\log X}\right) \Phi\left( \frac {d}{X} \right), \]
so that $S(X,Y; \hat{\phi},f, \Phi)=S_M(X,Y; \hat{\phi},
f,\Phi)+S_R(X,Y; \hat{\phi},f, \Phi)$. \newline

    Using standard techniques (see \eqref{error1}), we can show that by choosing $Z$ appropriately that $S_R(X,Y;
    \hat{\phi},f,
    \Phi)$ is small.
    Hence the main term arises only from
    $S_M(X,Y; \hat{\phi},f, \Phi)$. We write
it as
\begin{align}
\label{3.0}
    S_{M}(X,Y; \hat{\phi}, f, \Phi)  =  \sum_{p \leq Y}
    \frac {a_f(p)\log p}{\sqrt {p}} \left( \frac{8}{p} \right) \hat{\phi}\left( \frac {\log p}{2\log X} \right)
    \left ( \sum_{\gcd(d,2)=1}M_Z(d) \left( \frac{d}{p} \right)
    \Phi \left( \frac {d}{X} \right) \right).
\end{align}

  We now evaluate the inner sum above following a method of K. Soundararajan in \cite{sound1} by applying the Poisson summation
    formula to the sum over $d$. For all odd
    integers $k$ and all integers $m$, we introduce the Gauss-type
    sums
\begin{equation*}
\label{010}
    \tau_m(k) := \sum_{a \shortmod{k}}\left( \frac {a}{k} \right) e \left( \frac {am}{k} \right) =:
    \left( \frac {1+i}{2}+\left( \frac {-1}{k} \right)\frac {1-i}{2}\right) G_m(k).
\end{equation*}
We quote Lemma 2.3 of \cite{sound1} which determines $G_m(k)$.
\begin{lemma}
\label{lem1}
   If $(k_1,k_2)=1$ then $G_m(k_1k_2)=G_m(k_1)G_m(k_2)$. Suppose that $p^a$ is
   the largest power of $p$ dividing $m$ (put $a=\infty$ if $m=0$).
   Then for $b \geq 1$ we have
\begin{equation*}
\label{011}
    G_m(p^b)= \left\{\begin{array}{cl}
    0  & \mbox{if $b\leq a$ is odd}, \\
    \phi(p^b) & \mbox{if $b\leq a$ is even},  \\
    -p^a  & \mbox{if $b=a+1$ is even}, \\
    (\frac {m/p^a}{p})p^a\sqrt{p}  & \mbox{if $b=a+1$ is odd}, \\
    0  & \mbox{if $b \geq a+2$}.
    \end{array}\right.
\end{equation*}
\end{lemma}

   For a Schwartz function $F$, we define
\begin{equation} \label{tildedef}
   \tilde{F}(\xi)=\frac {1+i}{2}\hat{F}(\xi)+\frac
   {1-i}{2}\hat{F}(-\xi)=\int\limits^{\infty}_{-\infty}\left(\cos(2\pi \xi
   x)+\sin(2\pi \xi x) \right)F(x) \dif x.
\end{equation}
    We quote Lemma 2.6 of \cite{sound1} which determines the
   inner sum in \eqref{3.0}.
\begin{lemma}
\label{lem2}
   Let $\Phi$ be a non-negative, smooth function supported in
   $(1,2)$. For any odd integer $k$,
\begin{equation*}
\label{013}
  \sum_{\gcd(d,2)=1}M_Z(d)\left( \frac {d}{k} \right)
    \Phi\left( \frac {d}{X} \right)=\frac {X}{2k}\left( \frac {2}{k} \right) \sum_{\substack {\alpha \leq Z \\ \gcd(\alpha, 2k)=1}}\frac {\mu(\alpha)}{\alpha^2}
    \sum_m(-1)^mG_m(k)\tilde{\Phi}\left( \frac {mX}{2\alpha^2k} \right),
\end{equation*}
where $\tilde{\Phi}$ is as defined in \eqref{tildedef}.
\end{lemma}

    Note that for any non-negative integer $l$,
\begin{equation*}
    \tilde{\Phi}^{(l)}(\xi)\ll 1, \hspace{0.1in} |\xi| < 1.
\end{equation*}
     Also note via integration by parts,
\begin{equation*}
\label{17}
    \tilde{\Phi}(\xi)=\frac {-1}{2\pi \xi}\Bigl(\int\limits^{1+1/U}_{1}+\int\limits^{2}_{2-1/U} \Bigr )
    \Phi'(x) \Bigl(\sin(2\pi \xi
   x)-\cos(2\pi \xi x) \Bigr ) \dif x \ll \frac 1{|\xi|}.
\end{equation*}
    Similarly, one can show for any $l \geq 0, j \geq 1$,
\begin{equation*}
    \tilde{\Phi}^{(l)}(\xi) \ll \frac {U^{j-1}}{|\xi|^j}.
\end{equation*}

\subsection{Primitive cubic and quartic Dirichlet characters}
\label{sec2.4}
   The classification of all the primitive cubic characters of conductor $q$ coprime to $3$ is given in
   \cite{B&Y}. It is shown there that every such character is of the form
   $m \rightarrow (\frac{m}{n})_3$ for some $n \in \mz[\omega]$, with $n \equiv 1 \pmod{3}$, $n$ square-free and not
divisible by any rational primes, $N(n) = q$. Here the symbol
$(\frac{\cdot}{n})_3$ is the cubic residue symbol in the ring
$\mz[\omega]$.  For a prime $\pi \in\mz[\omega]$ with $N(\pi) \neq
3$, the cubic character is defined for $a \in \mz[\omega]$,
$\gcd(a, \pi)=1$ by $\leg{a}{\pi}_3 \equiv a^{\frac{N(\pi)-1}{3}}
\pmod{\pi}$, with $\leg{a}{\pi}_3 \in \{ 1, \omega, \omega^2 \}$.
When $\pi | a$, set $\leg{a}{\pi}_3 =0$. One then extends the
cubic character to composite $n$ with $\gcd (N(n), 3)=1$
multiplicatively. \newline

Similarly, one can give a classification of all the primitive complex quartic characters of conductor $q$ coprime to
    $2$.  Every such character is of the form $m \rightarrow
(\frac{m}{n})_4$ for some $n \in \mz[i]$, with $n \equiv 1
\pmod{(1+i)^3}$, $n$ square-free and not divisible by any rational
primes and $N(n) = q$. Here the symbol $(\frac{\cdot}{n})_4$ is
the quartic residue symbol in the ring $\mz[i]$.  For a prime $\pi
\in \mz[i]$ with $N(\pi) \neq 2$, the quartic character is defined
for $a \in \mz[i]$, $\gcd(a, \pi)=1$ by $\leg{a}{\pi}_4 \equiv
a^{\frac{N(\pi)-1}{4}} \pmod{\pi}$, with $\leg{a}{\pi}_4 \in \{
\pm 1, \pm i \}$. When $\pi | a$, $\leg{a}{\pi}_4$ is defined to
be zero.  Then the quartic character can be extended to composite
$n$ with $\gcd(N(n), 2)=1$ multiplicatively.  Note that in
$\intz[i]$, every ideal coprime to $2$ has a unique generator
congruent to 1 modulo $(1+i)^3$.

\section{Proof of Theorem \ref{quadraticmainthm}}

    Note that as $X \rightarrow \infty$,
\begin{align*}
    \sum_{d \in D(X)}\Phi \left( \frac {d}{X} \right) \sim \# D(X).
\end{align*}
     Moreover, as ${\hat \phi}$ is supported in $(-1,1)$, we
     have
\begin{equation*}
  \int\limits^{\infty}_{-\infty}\phi(t) \dif t+\frac 1{2}
   \int\limits^{\infty}_{-\infty}\hat{\phi}(u) \dif u=\int\limits^{\infty}_{-\infty}\phi(t)\left( 1+\frac {\sin (2 \pi t)}{2 \pi
   t} \right) \dif t.
\end{equation*}
     Theorem \ref{quadraticmainthm} thus follows from \eqref{2.5}, provided that
     we show for any Schwartz function $\phi$ with $\hat{\phi}$ supported in
$(-1, 1)$,
\begin{equation}
\label{01.50}
  \lim_{X \rightarrow \infty} \frac{S(X, Y;\hat{\phi}, f, \Phi)}{X \log X}=0.
\end{equation}

As
\[ S(X,Y; \hat{\phi},f, \Phi)=S_M(X,Y; \hat{\phi}, f, \Phi)+S_R(X,Y; \hat{\phi},f, \Phi), \]
the remainder of this section is therefore devoted to the evaluation of $S_R(X,Y; \hat{\phi},f, \Phi)$ and $S_{M}(X,Y; {\hat \phi}, f, \Phi)$.
\subsection{Estimation of $S_R(X,Y; \hat{\phi},f, \Phi)$} \label{sec 3.1}
In this section, we estimate $S_R(X,Y; \hat{\phi},f, \Phi)$. We first seek a bound for
\begin{equation*}
  E(Y;\chi, \hat{\phi}, f) :=\sum_{p \leq Y}
    \frac {a_f(p)\log p}{\sqrt {p}}\chi (p)\hat{\phi} \left( \frac {\log p}{2\log X} \right),
\end{equation*}
   for any non-principal quadratic character $\chi$ with modulus $q$ and $Y \leq
   X^{2-2\epsilon}$. For this we need the following result which follows from \cite[Theorem
   5.15]{iwakow}.
\begin{lemma}
\label{lem3.2}
Suppose that GRH is true. For any Dirichlet character $\chi$ with modulus $q$, we have for $x \geq 1$,
\begin{equation*}
  \sum_{p \leq x} a_f(p)\chi (p)\log p
\ll x^{1/2}\log^{2} (qx).
\end{equation*}
\end{lemma}

   It follows from the above lemma and partial summation that
\begin{equation}
\label{3.1}
  E(Y;\chi, \hat{\phi}, f) \ll \log^{3} (qX).
\end{equation}

   Now on writing $d=l^2m $, we obtain
\begin{equation} \label{error1}
\begin{split}
    S_R(X,Y; \hat{\phi}, f, \Phi) &=
     \sum_{\substack{l>Z \\ \gcd(l, 2)=1}}\mu(l)\sum_{\gcd(m,2)=1}
     \Phi \left( \frac{l^2m}{X} \right) E(Y;\chi_{8l^2m}, \hat{\phi}, f)
     \\
     &\ll \sum_{l>Z}\sum_{X/l^2 \leq m \leq
   2X/l^2} \log^{3}(X) \ll \frac {X \log^{3} X}{Z}.
\end{split}
\end{equation}

\subsection{Estimation of $S_{M}(X,Y; {\hat \phi}, f, \Phi)$}

    Applying Lemma \ref{lem2} to the inner sum of \eqref{3.0}, we see that the sum in $S_{M}(X,Y; {\hat \phi}, f, \Phi)$ corresponding to the contribution of  $m=0$ is zero, as it follows directly from the definition that $G_0(k)=\varphi(k)$ if $k$ is a square and $G_0(k)=0$ otherwise. \newline

    Now, the sums in $S_{M}(X,Y; {\hat \phi}, f, \Phi)$
corresponding to the contribution of $m \neq 0$ can be written as
$XR/2$, where
\begin{align*}
  R= \sum_{\substack {p \leq Y \\ \gcd (2, p)=1}}
   \frac {a_f(p)\log p}{p}
   \hat{\phi} \left( \frac {\log p}{2\log X} \right) \sum_{\substack { \alpha \leq Z \\ \gcd(\alpha, 2p)=1}}\frac
{\mu(\alpha)}{\alpha^2}
    \sum^{\infty}_{\substack{m=-\infty \\ m \neq 0}} \left( \frac{m}{p} \right) (-1)^m\tilde{\Phi}
   \left(  \frac{m X}{2\alpha^2p} \right).
\end{align*}

   We recast the condition $\gcd(p, 2\alpha)=1$ as
   $\chi_{4\alpha^2}(p)$ and use estimation \eqref{3.1} to deduce by partial
   summation (note that we will take $Z$ to be smaller than
   some power of $X$ so that $\log \alpha \ll \log X$) that
\begin{equation*}
\begin{split}
\sum_{\substack
{p \leq Y \\ \gcd (2\alpha, p)=1}} & \frac {a_f(p)\log p}{p}
  \hat{\phi}\left( \frac {\log p}{2\log X} \right) \tilde{\Phi}
    \left( \frac {m X}{2\alpha^2p} \right)
    \left( \frac{m}{p} \right) =\int\limits^{Y}_1\frac {1}{\sqrt{V}}\tilde{\Phi}
    \left( \frac {m X}{2\alpha^2V} \right) \dif E(V;\chi_{4\alpha^2m}, \hat{\phi}, f)
    \\
   & \ll \log^{3}(X(|m|+2))
    \left( \frac {1}{\sqrt{Y}}
   \left |\tilde{\Phi}\left (\frac {m
   X}{2\alpha^2Y} \right ) \right |  +  \int\limits^{Y}_{1}\frac
   {1}{V^{3/2}}\left |\tilde{\Phi}
    \left (\frac {m X}{2\alpha^2V} \right ) \right |  \dif V \right. \\
   & \hspace*{2in} \left. + \int\limits^{Y}_{1}\frac
   {X}{\alpha^2V^{5/2}} \left |m\tilde{\Phi}'
    \left(\frac {m X}{2\alpha^2V} \right ) \right | \dif V \right).
\end{split}
\end{equation*}

    Hence we have
\begin{equation*}
  R  \ll \sum_{\alpha \leq Z}\frac {1}{\alpha^2}(R_1+R_2+R_3),
\end{equation*}
    where
\[ R_1 = \frac {1}{\sqrt{Y}} \sum_{m \neq 0}\log^{3}(X(|m|+2))\left |\tilde{\Phi}\left (\frac {m
   X}{2\alpha^2Y}\right ) \right |, \;  R_2 = \int\limits^{Y}_{1}\frac
   {1}{V^{3/2}}\sum_{m \neq 0}\log^{3}(X(|m|+2)) \left |\tilde{\Phi}
    \left (\frac {m X}{2\alpha^2V} \right ) \right | \dif V \]
and
\[   R_3 = \int\limits^{Y}_{1}\frac
   {X}{\alpha^2V^{5/2}}\sum_{m \neq 0}\log^{3}(X(|m|+2)) \left|m\tilde{\Phi}'
    \left(\frac {m X}{2\alpha^2V} \right) \right | \dif V. \]

We now gather the estimates in Appendix C of \cite{Miller1} for $R_1, R_2$ and $R_3$ (but be aware that the sum over $p$ in \cite{Miller1} does not exist in our situation here):
\begin{equation} \label{Rest}
  R_1 + R_2 \ll \frac {U \alpha^2\sqrt{Y}\log^{7} X}{X} \; \mbox{and} \; R_3 \ll \frac {U \alpha^2\sqrt{Y}\log^{7} X}{X}+\frac {U^2 \alpha^4Y^{3/2}\log^{7} X}{X^{2008}}.
\end{equation}
Some of the estimates quoted above from \cite{Miller1} have their origins in \cite{Gao}.  The estimates from \cite{Miller1} would suffice for our purpose, but the improved estimates from \cite{Gao} for $R_1+R_2$ are neater to use here.  Combining these estimates in \eqref{Rest}, we obtain
\begin{equation*}
  R \ll \frac {UZ \sqrt{Y}\log^{7} X}{X}.
\end{equation*}

Thus we conclude that the contribution of $m \neq 0$ is bounded by
\begin{equation}
\label{error4}
    UZ \sqrt{Y}\log^{7} X.
\end{equation}

\subsection{Conclusion }  We now combine the bounds \eqref{error1},
\eqref{error4} and take $Y=X^{2-2\epsilon}, Z=\log^3 X$ (recall
that $U=\log X$) with any fixed $\epsilon>0$ to obtain
\[ S(X, Y;\hat{\phi}, f, \Phi) \ll \frac {X \log^{3} X}{Z} +UZ \sqrt{Y}\log^{7} X  =o(X\log X), \]
which implies \eqref{01.50} and this completes the proof of Theorem \ref{quadraticmainthm}.

\section{Proof of Theorem \ref{mainthm}}

Let $N_d(X)$ be the number of primitive cubic characters of
conductor $q \leq X$ with $\gcd (q, d) = 1$.  It is shown in
\cite{David} that $N_d(X) \sim c(d)X$ as $X \rightarrow \infty$
for some constant $c(d)$. It follows from this that $\# C(X) \sim
cX$ for some constant $c$ as $X \rightarrow \infty$. Combining
this with \eqref{01.3}, we see that in order to establish
Theorem~\ref{mainthm}, it suffices to show that for any Schwartz
function $f$ with $\hat{f}$ supported in $(-3/7, 3/7)$,
\begin{equation} \label{01.5}
\lim_{X \rightarrow \infty}  \frac {1}{X\log X}\sum_{p}\frac {\log p}{\sqrt{p}}\hat{f}\left( \frac {\log
p}{\log X} \right) \sum_{X \leq q \leq 2X} \ \sumstar_{\substack{\chi
\shortmod{q} \\ \chi^3 = \chi_0}} \left( \chi(p)+\overline{\chi}(p) \right) = 0
\end{equation}
and
\begin{equation} \label{primesquare}
\lim_{X \rightarrow \infty} \frac {1}{X\log X}\sum_{p}\frac
{2\log p}{p}\hat{f}\left( \frac {2\log p}{\log X}\right) \sum_{X \leq q \leq
2X} \ \sumstar_{\substack{\chi \shortmod{q} \\ \chi^3 =
\chi_0}} \left( \chi(p^2)+\overline{\chi}(p^2) \right)  = 0.
\end{equation}
   As both $\chi$ and $\bar{\chi}$ are
   primitive cubic characters, it is enough to consider the two limits for $\chi$ only.
   The term $p=3$ in each sum above is $O(X)$. Hence we may assume $p \neq 3$ in the sums above and we apply the Cauchy-Schwarz inequality to see that
\begin{equation} \label{estimation1}
\begin{split}
\sum_{p \neq 3}\frac {\log p}{\sqrt{p}}& \hat{f} \left( \frac {\log
p}{\log X} \right) \sum_{X \leq q \leq 2X} \ \sumstar_{\substack{\chi
\shortmod{q} \\ \chi^3 = \chi_0}} \chi(p) \\
&\ll \left( \sum_{p \leq
X^{1/5}}\frac {\log^2 p}{p} \right)^{1/2} \left(\sum_{3 \neq p \leq
X^{1/5}} \left|\sum_{X \leq q \leq 2X} \ \sumstar_{\substack{\chi
\shortmod{q} \\ \chi^3 = \chi_0}} \chi(p) \right|^2 \right)^{1/2},
\end{split}
\end{equation}
and
\begin{equation} \label{estimation2}
\begin{split}
\sum_{p \neq 3} & \frac {\log
p}{p}\hat{f} \left( \frac {2\log p}{\log
X} \right) \sum_{X \leq q \leq 2X} \ \sumstar_{\substack{\chi \shortmod{q} \\
\chi^3 = \chi_0}} \chi(p^2) \\
& \ll \left(\sum_{p \leq X^{1/5}}\frac
{\log^2 p}{p^2} \right)^{1/2} \left( \sum_{3 \neq p \leq X^{1/5}}
\left| \sum_{X \leq q \leq 2X} \ \sumstar_{\substack{\chi \shortmod{q} \\
\chi^3 = \chi_0}} \chi(p^2) \right|^2 \right)^{1/2}.
\end{split}
\end{equation}
   It's easy to see that
\begin{equation*}
   \sum_{p \leq
X^{1/5}}\frac {\log^2 p}{p} \ll \log^2 X \; \mbox{and} \; \sum_{p
\leq X^{1/5}}\frac {\log^2 p}{p^2} \ll 1.
\end{equation*}
   Moreover, note that for a primitive cubic character $\chi$,
   $\chi(p^2)=\bar{\chi}(p)$ which implies that the values of the sums on the right-hand sides of \eqref{estimation1} and \eqref{estimation2} are the same.  Hence, it remains to estimate
\begin{equation} \label{neweq}
   \sum_{3 \neq p \leq Y} \left|\sum_{X \leq q \leq 2X} \ \sumstar_{\substack{\chi
\shortmod{q} \\ \chi^3 = \chi_0}} \chi(p) \right|^2.
\end{equation}
   Here, $Y$ is a parameter independent of $X$. From our discussion in Section \ref{sec2.4}, we can recast \eqref{neweq} in terms of the cubic
residue symbol as
\begin{equation*}
   \sum_{3 \neq p \leq Y} \left| \sumprime_{\substack{N(n) \in \mathcal{I}(X)
   \\ n \equiv 1\shortmod{3}}}\left(\frac{p}{n}\right)_3 \right|^2,
\end{equation*}
where the inner sum runs over square-free elements $n$ of
$\mathbb{Z}[\omega]$ that have no rational prime divisor and
$\mathcal{I}(Z)$ henceforth denotes the dyadic interval
\[ \mathcal{I}(Z) := [Z, 2Z] \]
for any real number $Z$. We now regard $(\frac {p}{\cdot})_3$ as a
ray class group character $\xi$ on $h_{(3p)}$ where we define
$\xi((n))=(\frac {p}{n})_3$.
   Now we remove the condition that $n$ has no rational prime divisor by using M\"{o}bius inversion
   (note that one can uniquely express any $n \in \mz[\omega]$ as $n= n_1 n_2$, where $n_1 \in \mathbb{N}$,
   and $n_2$ has no rational prime divisor), obtaining
\begin{equation*}
\sumprime_{\substack{N(n) \in \mathcal{I}(X) \\ n
 \equiv 1\shortmod{3}}}\left(\frac {p}{n}\right)_3 = \sum_{\substack{d \in \mathbb{Z}, d^2 \leq 2X \\ d \equiv 1\shortmod{3} \\ \gcd(d, p) =
1}} \mu_{\mn}(|d|) \sumprimeprime_{\substack{N(n) \in
\mathcal{I}(X/d^2)
\\ n \equiv 1 \shortmod{3}}} \left(\frac {p}{n}\right)_3 ,
\end{equation*}
where the double prime indicates that $nd$ is square-free (viewed
as an element of $\mz[\omega]$).  Here $\mu_{\mn}$ is the usual
M\"{o}bius function defined on $\mn$. Note as $d$ and $p$ are
coprime rational integers, it follows from the corollary to
Proposition 9.3.4. of \cite{I&R} that $(\frac {p}{d})_3=1$.
\newline

Since $d$ is automatically square-free (as an element of
$\mz[\omega]$), $nd$ being square-free simply means that $n$ is
square-free and $\gcd(n,d) = 1$. Now use M\"{o}bius again (writing
$\mu_{\omega}$ for the M\"{o}bius function on $\mz[\omega]$) to
detect the condition that $n$ is square-free, getting
\begin{equation}
\label{eq:f}
 \sumprime_{\substack{X \leq N(n) \leq 2X
   \\ n
 \equiv 1\shortmod{3}}}\left(\frac {p}{n} \right)_3 = \sum_{\substack{d \in \mathbb{Z}, d^2 \leq 2X \\ d \equiv 1\shortmod{3} \\  \gcd (d,p) = 1}}
 \mu_{\mn}(|d|) \sum_{\substack{N(l)^2 \leq 2X/d^2 \\ \gcd(l, d) = 1 \\ l \equiv 1 \shortmod{3}}} \mu_{\omega}(l)\left(\frac {p}{l^2}\right)_3
\sum_{\substack{N(n) \in \mathcal{I}(X/(N(l)d)^2) \\ \gcd(n, d) = 1 \\
n\equiv 1 \shortmod{3}}} \left(\frac {p}{n} \right)_3 .
\end{equation}
Here we changed variables via $n = l^2 n'$ and fixed $l$ up to
a unit by the condition $l \equiv 1 \hspace{0.1in} \shortmod{3}$
(note that in $\intz[\omega]$, every ideal coprime to $3$ has a
unique generator congruent to 1 modulo 3). \newline

We now apply a further M\"{o}bius to remove the condition $\gcd(n, d) =1$ to get
\begin{equation} \label{3.4}
\sum_{\substack{N(n) \in \mathcal{I}(X/(N(l)d)^2) \\ \gcd(n, d)
= 1 \\
n\equiv 1 \shortmod{3} }} \left( \frac {p}{n} \right)_3
=\sum_{\substack{e|d \\ e \equiv 1 \shortmod{3}}}
\mu_{\omega}(e) \left( \frac {p}{e} \right)_3\sum_{\substack{N(n) \in \mathcal{I}(X/N(e)N^2(l)d^2) \\
n\equiv 1 \shortmod{3}}} \left( \frac {p}{n} \right)_3.
\end{equation}
Now we need the following lemma, which establishes a
   P\'olya-Vinogradov type inequality for the cubic symbols:
\begin{lemma}
\label{PVbounds} Let $p \neq 3$ be a rational prime. Then we have
\begin{align}
\label{4.5}
\sum_{\substack{ N(n) \leq X  \\
n\equiv 1 \shortmod{3}}} \left(\frac {p}{n}\right)_3 \ll
 X^{1/3}p^{2/3} \log^2 p,
\end{align}
   where the sum runs over elements $n \in \mathbb{Z}[\omega]$.
\end{lemma}
\begin{proof}
  As we mentioned above, we regard $(\frac {p}{\cdot})_3$ as a
ray class group character $\xi$ on $h_{(3p)}$ so that we can
recast the sum in \eqref{4.5} as
\begin{equation}
\label{4.6}
 \sum_{\substack{ N(n) \leq X  \\
n\equiv 1 \shortmod{3}}} \left(\frac {p}{n} \right)_3 =
\sum_{\substack{N(I) \leq X \\ \gcd(I, 3)=1}} \xi(I),
\end{equation}
   where the sum above runs over non-zero integral ideals $I \in
   \mz[\omega]$.
 It's easy to see that $\xi$ is induced by a
primitive character of conductor $(ap)$ for some $a | 3$.
Therefore, the condition $\gcd(I,3)=1$ imposed on the sum on the
right-hand side of \eqref{4.6}
   implies that the said sum remains unchanged if $\xi$ is replaced by $\xi^*$ (say), the primitive character that induces $\xi$.  We may
 therefore assume without loss of generality that $\xi$
   is primitive and we further use M\"{o}bius to detect the condition
$\gcd(I,3)=1$ in the second sum in \eqref{4.6} while noting that the
only ideals dividing $3$ are $(1)$, $(1-\omega)$ and $(3)$ to get
\begin{equation*}
\sum_{\substack{ N(n) \leq X  \\
n\equiv 1 \shortmod{3}}} \left(\frac {p}{n}\right)_3 =
\sum_{\substack{N(I) \leq X
\\ \gcd(I, 3)=1}} \xi (I)=\sum_{h = 1, 1-\omega, 3}
\mu_{\omega}(h) \xi (h)\sum_{N(I) \leq X/N(h)} \xi (I).
\end{equation*}

Now we quote a result of E. Landau \cite{Landau} (see also
\cite[Theorem 2]{Sunley}), which states that for an algebraic
number field $K$ of degree $n \geq 2$,
   $\xi$ any primitive ideal character of $K$ with conductor $\mathfrak{f}$, $k =|N(\mathfrak{f}) \cdot d_K|$ with $d_K$ being the
discriminant of $K$, we have for $X \geq 1$,
\begin{equation*}
   \sum_{N(I) \leq X}\xi(I) \leq k^{1/(n+1)}\log^n(k) \cdot X^{(n-1)/(n+1)},
\end{equation*}
where $I$ runs over integral ideas of $K$. \newline

We now identify $K=\ratq(\omega)$ with $n=2$ and $k=3ap^2$, $a=1$,
$3$ or $9$ to see that the sum on the right-hand side of
\eqref{4.6} is
\[ O \left( X^{1/3}p^{2/3}\log^2 p \right). \]
  This now completes the proof of the lemma.
\end{proof}

Applying Lemma \ref{PVbounds}, we can majorize the left-hand side expression in \eqref{3.4} as
\begin{align} \label{afterPV}
  \sum_{\substack{N(n) \in \mathcal{I}(X/(N(l)d)^2) \\ \gcd(n, d) = 1 \\
n\equiv 1 \shortmod{3} }} \left(\frac {p}{n}\right)_3 \ll \frac
{X^{1/3}p^{2/3}\log^2 p}{(N(l)d)^{2/3}} \sum_{e|d }\frac
{1}{\sqrt{N(e)}} \ll \frac {X^{1/3}p^{2/3}\log^2
p}{(N(l))^{2/3}d^{2/3-2\epsilon}},
\end{align}
  for any $\epsilon>0$. The last bound follows since we have $\#\{e \in
   \mathbb{Z}[\omega]: e|d \} \ll N(d)^{\epsilon}$.  From \eqref{afterPV} and \eqref{eq:f}, we have
\begin{equation*}
 \sumprime_{\substack{N(n) \in \mathcal{I}(X) \\ n
 \equiv 1\shortmod{3}}} \left( \frac {p}{n} \right)_3 \ll X^{1/3}p^{2/3}\log^2
p
 \sum_{\substack{d \in \mathbb{N}\\ d^2 \leq 2X}}\frac {1}{d^{2/3-2\epsilon}}
\sum_{\substack{N(l)^2 \leq 2X/d^2 \\ l \equiv 1
\shortmod{3}}}\frac {1}{(N(l))^{2/3}}.
\end{equation*}
   Note that it follows from \cite[Theorem 2]{Sunley} that
\begin{equation*}
   \sum_{\substack{N(l) \leq X \\ l \equiv 1
\shortmod{3}}}1 \ll X.
\end{equation*}
   We then deduce by partial summation that
\begin{equation*}
 \sum_{\substack{N(l)^2 \leq 2X/d^2 \\ l \equiv 1
\shortmod{3}}}\frac {1}{(N(l))^{2/3}} \ll \Big (\frac {X}{d^2}\Big
)^{1/6}.
\end{equation*}
   Therefore
\begin{equation*}
 \sumprime_{\substack{N(n) \in \mathcal{I}(X) \\ n
 \equiv 1\shortmod{3}}}\left( \frac {p}{n} \right)_3 \ll X^{1/2}p^{2/3}\log^2
p
 \sum_{\substack{d \in \mathbb{N}\\ d^2 \leq 2X}}\frac {1}{d^{1-2\epsilon}} \ll X^{1/2+3\epsilon}p^{2/3}\log^2 p
 .
\end{equation*}
Hence, we obtain
\begin{equation*}
   \sum_{3 \neq p \leq Y} \left|\sumprime_{\substack{N(n) \in \mathcal{I}(X)
   \\ n
 \equiv 1\shortmod{3}}}\left(\frac {p}{n} \right)_3 \right|^2 \ll X^{1+6\epsilon}\sum_{p \leq Y}p^{4/3}\log^4 p \ll X^{1+6\epsilon}Y^{7/3}\log^3 Y.
\end{equation*}
    Applying the above bound in the estimations \eqref{estimation1}
    and \eqref{estimation2}, we find that \eqref{01.5} and \eqref{primesquare} hold so
    long as $Y^{7/6} \leq X^{1/2-7\epsilon/2}$ and as $\epsilon$
    is arbitrary, the proof of Theorem
    \ref{mainthm} is completed.

\section{Proof of Theorem \ref{mainthmonK} and Theorem \ref{quarticthm}}

    The proofs of both Theorems \ref{mainthmonK} and
    \ref{quarticthm} are similar to that of Theorem
    \ref{mainthm} so we shall skip most of the details. For the proof of Theorem \ref{mainthmonK}, one can show, similar to the proof of \cite[Lemma 4.2]{Gu} that
\[ \# C_{(9)}(X) \sim c'X \; \mbox{as} \; X \to\infty\]
for some constant $c'$.  Recall that $C_{(9)}$ denotes the set of
cubic symbols $\chi_c = (\frac {\cdot}{c})_3$ with $c$
square-free, congruent to 1 modulo 9 and $X \leq N(c) \leq 2X$. We
then proceed as in the proof of Theorem \ref{mainthm} to see that
it suffices to show for any fixed $\epsilon>0$, we have
\begin{align} \label{cubicest}
  & \sum_{\substack{\mathfrak{p} \\3 \neq N(\mathfrak{p}) \leq Y}}
   \left| \ \sumstar_{\substack{N(c) \in \mathcal{I}(X) \\ c \equiv 1 \shortmod{9}}} \left(\frac {\mathfrak{p}}{c}\right)_3 \right|^2
   \ll X^{1+\epsilon}Y^{5/3}\log^3 Y,
\end{align}
where $Y$ is a parameter independent of $X$.  We now regard
$(\frac {\mathfrak{p}}{\cdot})_3$ as a ray class group character
$\xi$ on $h_{(3)\mathfrak{p}}$ where we set $\xi((c))=(\frac
{\mathfrak{p}}{c})_3$. Using the ray class characters on
$h_{(3)\mathfrak{p}}$ to detect the condition $c \equiv 1
\pmod{9}$ in the inner sum on the left-hand side of
\eqref{cubicest}, we get
\begin{equation*}
   \ \sumstar_{\substack{N(c) \in \mathcal{I}(X) \\ c \equiv 1 \shortmod{9}}} \left(\frac {\mathfrak{p}}{c}\right)_3 =
    \frac{1}{\#h_{(9)}} \sum_{\psi \shortmod{9}} \ \sumstar_{\substack{N(c) \in \mathcal{I}(X) \\ c \equiv 1 \shortmod{3}}}
    \psi((c))\xi((c)).
\end{equation*}

The estimation in \eqref{cubicest} follows, after using M\"{o}bius
to detect the condition that $c$ is square-free, from the
following estimation:
\[ \sum_{\substack{N(c) \leq X \\ c \equiv l \shortmod{9}}}
    \psi((c))\xi((c))
\ll  X^{1/3}N(\mathfrak{p})^{1/3} \log^2 N(\mathfrak{p}) \]
     where $l \equiv 1 \hspace{0.1in} \shortmod{3} \in
     \intz[\omega]$. One more application of the ray class characters on
$h_{(9)}$ shows that the above estimation follows from the
following:
\begin{equation}
\label{4.3}
  \sum_{\substack{N(I) \leq X \\ \gcd(I, 3)=1}}
    \psi(I)\psi'(I)\xi(I)
\ll  X^{1/3}N(\mathfrak{p})^{1/3} \log^2 N(\mathfrak{p}),
\end{equation}
   where the sum above runs over non-zero integral ideals $I \in
   \mz[\omega]$. The character $\psi \psi' \xi$ can be viewed as
   a ray class group character on $h_{(9)\mathfrak{p}}$ and our
   definition of $\xi$ implies that it is induced from a
   primitive character on $h_{(a)\mathfrak{p}}$ with $a|9$.  Consequently, as the
   condition $\gcd(I,3)=1$ is imposed on the summation in
   \eqref{4.3}, the value of this sum remains unaltered if $\psi \psi'\xi $ is replaced by the primitive character that induces it. We may
   therefore without loss of generality assuming $\psi \psi' \xi$
   is primitive and an application of the M\"obius function as in the
   proof of Lemma \ref{PVbounds} allows us to obtain the desired
   bound in \eqref{4.3}. \newline

For the proof of Theorem \ref{quarticthm}, one can show, following
the approach in \cite{David}, that
\[ \# Q(X)  \sim dX \; \mbox{as} \; X \to \infty \]
for some constant $d$ as $X \rightarrow \infty$.  The rest of the
proof goes in a similar fashion as that of Theorem \ref{mainthm}.

\section{Notes}

We remark here that it is conceivable that results along the lines of Theorems \ref{mainthm}, \ref{mainthmonK} and \ref{quarticthm} can be proved using a simpler approach involving mean-value estimates for sums of characters of a fixed order.  This method was used in \cite{Miller1}.  The afore-mentioned mean-value estimate for quadratic character sums is due to M. Jutila \cite[Lemma 5]{Jut}, but the analogous results for characters of orders higher than two which would be needed here do not seem to be available.  Moreover, if good mean value estimates can be obtained for cubic and quartic characters, one would expect that the support of $\hat{f}$ in Theorems \ref{mainthm}, \ref{mainthmonK} and \ref{quarticthm} can be significantly widened. \newline

It would also be interesting to consider the one level density of low-lying zeros of families of Dirichlet $L$-function for characters of orders larger than 4.  But the relation between higher order residue symbols and $n$-th order primitive Dirichlet characters would be more difficult to describe. \newline

\subsection*{Acknowledgments}  During this work, P. G. was
supported by some postdoctoral research fellowships at Nanyang
Technological University (NTU) and L. Z. by an Academic Research
Fund Tier 1 Grant at NTU.  Both authors wish to thank Stephan
Baier for some interesting discussions concerning some parts of
this paper and Matt Young for making them aware of the estimates
in \cite{Landau} which led to some improvements.  Finally, the authors would also like to thank the referee of his/her many comments and suggestions.

\bibliography{biblio}
\bibliographystyle{amsxport}

\vspace*{.5cm}

\noindent\begin{tabular}{p{8cm}p{8cm}}
Div. of Math. Sci., School of Phys. \& Math. Sci., & Div. of Math. Sci., School of Phys. \& Math. Sci., \\
Nanyang Technological Univ., Singapore 637371 & Nanyang Technological Univ., Singapore 637371 \\
Email: {\tt penggao@ntu.edu.sg} & Email: {\tt lzhao@pmail.ntu.edu.sg} \\
\end{tabular}

\end{document}